\documentclass[12pt]{amsart}
\usepackage{amssymb,verbatim,amscd,amsmath,graphicx}
\usepackage{graphicx}
\usepackage{caption}
\usepackage{subcaption}
\usepackage{pdfsync}

\numberwithin{equation}{section}
\newtheorem{theorem}{Theorem}[section]
\newtheorem{lemma}[theorem]{Lemma}
\newtheorem{proposition}[theorem]{Proposition}
\newtheorem{corollary}[theorem]{Corollary}

\theoremstyle{definition}

\newtheorem{def-prop}[theorem]{Definition-Proposition}

\newtheorem{remark}[theorem]{Remark}
\newtheorem{example}[theorem]{Example}

\DeclareMathOperator{\reg}{reg}
\DeclareMathOperator{\depth}{depth}

\DeclareMathOperator{\Ass}{Ass}

\DeclareMathOperator{\lin}{lin}

\newcommand{\ZZ}{{\mathbb Z}}
\newcommand{\NN}{{\mathbb N}}

\def\mm{{\mathfrak m}}
\def\nn{{\mathfrak n}}

\def\MM{{\mathfrak M}}

\def\1{{\bf 1}}
\def\0{{\bf 0}}

\begin{document}


\title{Depth and regularity of powers of sums of ideals}

\author{Huy T\`ai H\`a}
\address{Tulane University \\ Department of Mathematics \\
6823 St. Charles Ave. \\ New Orleans, LA 70118, USA}
\email{tha@tulane.edu}
\urladdr{http://www.math.tulane.edu/$\sim$tai/}

\author{Ngo Viet Trung}
\address{Institute of Mathematics \\ Vietnam Academy of Science and Technology \\ 18 Hoang Quoc Viet \\ Hanoi, Vietnam}
\email{nvtrung@math.ac.vn}
\urladdr{http://math.ac.vn/en/component/staff/?task=getProfile\&staffID=64}

\author{Tran Nam Trung}
\address{Institute of Mathematics \\ Vietnam Academy of Science and Technology, \\ 18 Hoang Quoc Viet \\ Hanoi, Vietnam}
\email{tntrung@math.ac.vn}
\urladdr{http://math.ac.vn/en/component/staff/?task=getProfile\&staffID=65}

\thanks{We would like to thank Vietnam Institute for Advanced Study in Mathematics  for the hospitality during our visit in 2014, when we started to work on this paper. The first named author is partially supported by the Simons Foundation (grant \#279786). The second author  is supported by Vietnam National Foundation for Science and Technology Development under grant number 101.04-2014.52.}

\begin{abstract}
Given arbitrary homogeneous ideals $I$ and $J$ in polynomial rings $A$ and $B$ over a field $k$, we investigate the depth and the Castelnuovo-Mumford regularity of powers of the sum $I+J$ in $A \otimes_k B$ in terms of those of $I$ and $J$. Our results can be used to study the behavior of the depth and regularity functions of powers of an ideal. For instance, we show that such a depth function can take as its values any infinite non-increasing sequence of non-negative integers.
\end{abstract}
\maketitle


\section*{Introduction}

The motivation of our work is the following general problem: given an ideal $Q$ and a natural number $n$, what can be said about the power $Q^n$? Asymptotically, as $n$ gets large, $Q^n$ often exhibits nice behaviors. However, for small or particular values of $n$, there is no general approach to investigate properties of $Q^n$. To address this problem, one usually considers special classes of ideals that possess further structures.
An instance is the class of squarefree monomial ideals which are associated to combinatorial objects, such as (hyper)graphs and simplicial complexes; these are unions of disconnected structures that correspond to sums of ideals in different variables. We are naturally led to the following question.  \medskip\par

\noindent{\bf Question 1}.
Let $A = k[x_1, \dots, x_r]$ and $B = k[y_1, \dots, y_s]$ be polynomial rings over a field $k$, and let $R = k[x_1, \dots, x_r,y_1, \dots, y_s]$. Let $I \subset A$ and $J \subset B$ be nonzero proper  ideals. What properties of $(I+J)^n \subset R$ can be described in terms of those of $I$ and $J$?
\medskip\par

Question 1 is also related to the study of singularities of the fiber product $X \times_k Y$ of two algebraic schemes $X$ and $Y$,
which are encoded in properties of powers of the defining ideal.
As far as we know, this seemingly basic question has not yet been addressed for arbitrary ideals $I$ and $J$. \par

In this paper, we initiate our study of Question 1 by focusing on the depth and the Castelnuovo-Mumford regularity of  $(I+J)^n$ when $I$ and $J$ are homogeneous ideals. Depth and Castelnuovo-Mumford regularity of powers of ideals have been the main objects of studies of various authors in the last few decades. For instance, see \cite{HaS, HH, HQ, HV, NV, TT1, TT2,Tr} for studies on depth of powers of ideals, and see \cite{ACH, Be, Ch, Co, CHT, EH, EU, Ha, Ko, Sw, TW} for studies on regularity of powers of ideals. Aside from giving answers to Question 1, our results also contribute to a better understanding of the asymptotic behavior of depth and regularity.  \par

Our first main result gives the following bounds for the depth and the regularity of $R/(I+J)^n$.
\smallskip\par

\noindent{\bf Theorem \ref{first bound}.}
For all $n \ge 1$, we have \par
{\rm (i)} $\depth R\big/(I+J)^n \ge$
$$\hspace*{5ex}\min_{i \in [1,n-1], \ j \in [1,n]} \{\depth A/I^{n-i} + \depth B/J^i + 1, \depth A/I^{n-j+1} + \depth B/J^j\},$$ \par
{\rm (ii)} $\reg R\big/(I+J)^n \le$
$$\max_{i \in [1,n-1], \ j \in [1,n]} \{\reg A/I^{n-i} + \reg B/J^i + 1, \reg A/I^{n-j+1} + \reg B/J^j \}.$$

The proof of Theorem \ref{first bound} is based on an approximation of $(I+J)^n$ by a sequence of subideals whose successive quotients have good representations. The depth and the regularity of these quotients
can be expressed in terms of those of $I$ and $J$ by a recent work of Hoa and Tam on mixed products of ideals
\cite{HT}. We also present results showing that the minimum and maximum values of the two terms in the bounds of Theorem \ref{first bound} are attainable (Propositions \ref{example 1}, \ref{example 2}, and \ref{variable}). These results particularly show that there are no general formulae for the depth and regularity of $R/(I+J)^n$.  \par

On the other hand, our next main result gives the following general formulae for the depth and regularity of $(I+J)^n/(I+J)^{n+1}$.
\medskip\par

\noindent{\bf Theorem \ref{equality}.}
For all  $n \ge 1$, we have\par
{\rm (i)} $\displaystyle \depth (I+J)^n/(I+J)^{n+1} = \min_{i+j=n}\{\depth I^i/I^{i+1} + \depth J^j/J^{j+1}\},$\par
{\rm (ii)} $\displaystyle \reg (I+J)^n/(I+J)^{n+1} = \max_{i+j=n}\{\reg  I^i/I^{i+1} + \reg J^j/J^{j+1}\}.$
\medskip\par

Theorem \ref{equality} is a consequence of the decomposition of $(I+J)^n/(I+J)^{n+1}$ as a direct sum of modules of the form $I^i/I^{i+1} \otimes_kJ^j/J^{j+1}$, whose depth and regularity can be estimated by invoking a formula of Goto and Watanabe for the local cohomology of tensor products over a field \cite{GW}. \par

We can use Theorems \ref{first bound} and \ref{equality} to compute $\depth R/(I+J)^n$ and $\reg (I+J)^n$ for $n$ sufficiently large.
It is known that the depth (resp., the regularity) of powers of an ideal is asymptotically a constant (resp., a linear function); see for example \cite{Br, CHT, Ko}. However, the exact values of this constant and linear function remain mysterious despite much effort from various authors \cite{Ch,EH,EU,HH,HQ}. A na\"ive intuition might suggest that the constant for depth and the linear function for regularity of $(I+J)^n$ could be obtained by ``adding up'' corresponding invariants of $I$ and $J$. The following formulae show otherwise, and give an evidence as why the problem is difficult.
\medskip\par

\noindent{\bf Theorem \ref{lim}.}
$\displaystyle  \lim_{n\to \infty} \depth R/(I+J)^n =$\smallskip\par
\centerline{$\displaystyle  \min\big\{\lim_{i \to \infty}\depth A/I^i + \min_{j \ge 1}\depth B/J^j,\, \min_{i \ge 1} \depth A/I^i + \lim_{j \to \infty}\depth B/J^j\big\}.$}
\medskip\par

\noindent{\bf Theorem \ref{r-asymptotic}.}
Suppose that $\reg I^n = dn + e$ and $\reg J^n = cn+f$ for $n \gg 0$, where $c \ge d$.
Let $\lin(I)$ and $\lin(J)$ denote the least integer $m$ such that $\reg I^n = dn + e$ and $\reg J^n = cn+f$ for $n \ge m$, respectively.
Set $e^* = \max_{i \le \lin(I)}\{\reg I^i - ci\}$ and $f^* = \max_{j \le \lin(J)}\{ \reg J^j - dj\}$.
For $n \gg 0$ we have
$$\reg (I+J)^n =
\begin{cases} c(n+1)  + f+e^*-1 & \text{if } c > d,\\
 d(n+1) + \displaystyle \max\{f+e^*, e+f^*\}-1 & \text{if } c = d.
 \end{cases}$$

To prove Theorems \ref{lim} and \ref{r-asymptotic} we make use of the fact that for $n$ sufficiently large, $\depth R/(I+J)^n = \depth (I+J)^{n-1}/(I+J)^n$ and $\reg (I+J)^n = \reg (I+J)^{n-1}/(I+J)^n+1$. Here, the first equality is due to Herzog and Hibi \cite{HH}.
We also give an estimate for the least number $m$ such that  $\reg (I+J)^n$ becomes a linear function for $n \ge m$. However, we are unable to do the same for the index of stability of the function $\depth R/(I+J)^n$.
\par

Our results have interesting applications on the depth function of powers of an ideal. For example, we show that $R/(I+J)^i$ is Cohen-Macaulay for all $i \le n$ if and only if $A/I^i$ and $B/J^i$ are Cohen-Macaulay for all $i \le n$ (Proposition \ref{CM}). We also show that when $I, J$ are squarefree monomial ideals, $\depth R/(I+J)^n$ is a constant function if and only if so are $\depth A/I^n$ and $\depth B/I^n$ (Proposition \ref{HV}). This property improves a recent result of Herzog and Vladiou in \cite{HV} substantially. Furthermore, in \cite{HH}, Herzog and Hibi conjectured that the function $\depth R/Q^n$ of a monomial ideal $Q$ in a polynomial ring $R$ over $k$ can be any convergent non-negative integer valued function, and gave the proof for all non-decreasing functions and a large class of non-increasing functions. We complete the statement for all non-increasing functions (Theorem \ref{non-increasing}). \par

We should point out that even though stated for homogeneous ideals, our results on the depth in Theorem \ref{first bound}, Theorem \ref{equality}, and Theorem \ref{lim}  aso hold for non-homogeneous ideals $I,J$ if $A$, $B$, and $R$ are replaced by their localizations at maximal homogeneous ideals. \par

The paper is outlined as follows. In Section 1, we collect notations, terminology and basic results on depth and regularity. In Section 2, we give bounds for $\depth R/(I+J)^n$ and $\reg R/(I+J)^n$. The formulae for $\depth  (I+J)^n/(I+J)^{n+1}$ and $\reg (I+J)^n/(I+J)^{n+1}$ are proved in Section 3. Section 4 and Section 5 are devoted to the asymptotic behavior of $\depth R/(I+J)^n$ and $\reg (I+J)^n$, respectively.


\section{Preliminaries}

In this section, we collect notations, terminology and basic results used in the paper. We follow standard texts \cite{BH, E}.  \par

Let $(A,\mm)$ be a local ring and let $M$ be a finitely generated $A$-module.
An $M$-sequence is a sequence $z_1,\dots,z_r \in \mm$ such that
$(x_1,\dots,x_{i-1})M:x_i = (x_1,\dots,x_{i-1})M$
for $i = 1,\dots,r$. The {\em depth} of $M$, denoted by $\depth M$, is the length of a (or any) maximal $M$-sequence.
We always have $\depth M \le \dim M$, and if $\depth M = \dim M$ then $M$ is called a Cohen-Macaulay module.
\par

Let $H^i_\mm(M)$ denote the $i$th local cohomology module of $M$ with support $\mm$, for $i \ge 0$. By Grothendieck's theorem \cite[Theorem 3.5.7]{BH}, we have
\begin{align*}
\depth M = \min \{i ~|~ H^i_\mm(M) \not= 0\}.
\end{align*}
In this paper we shall use this characterization as the definition of $\depth M$.\par

Now, let $A$ be a standard graded algebra and let $\mm$ be the maximal homogeneous ideal of $A$. It is well-known that the notion of depth can also be defined similarly for finitely generated graded $A$-modules, and that these notions of depth over local and graded rings share the same properties and characterizations \cite[Remark 3.6.18]{BH}. This allows us to refer to depth of a module without specifying whether the underlying ring is local or graded. \par

In the graded setting, local cohomology modules $H^i_\mm(M)$ are also graded $A$-modules, and we can introduce the invariants $a^i(M) = \max \{t ~\big|~ [H^i_\mm(M)]_t \not= 0\}$, where by convention $a^i(M) = -\infty$ if $H^i_\mm(M) = 0$. The \emph{Castelnuovo-Mumford regularity} (or simply \emph{regularity}) of $M$ is defined by
$$\reg M := \max_{i\ge 0} \{a^i(M) +i\}.$$

If $A$ is a polynomial ring over a field then $\reg M$ controls the complexity of the graded structure of $M$ in the following sense. Let
$$0 \rightarrow \bigoplus_{j \in \ZZ} A(-j)^{\beta_{p,j}(M)} \rightarrow \cdots \rightarrow \bigoplus_{j \in \ZZ} A(-j)^{\beta_{0,j}(M)} \rightarrow M \rightarrow 0$$
be a minimal free resolution of $M$. Then
$$\reg(M) := \max \{j-i ~|~ \beta_{i,j}(M) \not= 0\}$$
by a result of Goto and Eisenbud \cite{EG}. \par

Given a short exact sequence of finitely generated graded $A$-modules
$$0 \rightarrow M \rightarrow N \rightarrow P \rightarrow 0,$$
by taking the derived long exact sequence of local cohomology modules, we can easily verify the following lemmas.

\begin{lemma} \label{depth} \quad
\begin{enumerate}
\item[(i)] $\depth N \ge \min \{\depth M, \depth P\},$
\item[(ii)] $\depth M \ge \min \{\depth N, \depth P+1\},$
\item[(iii)] $\depth P \ge \min \{\depth M-1, \depth N\},$
\item[(iv)] $\depth M = \depth P+1$ if  $\depth N > \depth P$,
\item[(v)] $\depth P = \depth M-1$ if  $\depth N > \depth M$,
\item[(vi)] $\depth P = \depth N$ if  $\depth N < \depth M$,
\item[(vii)] $\depth N = \depth M$ if  $\depth P \ge \depth M$.
\end{enumerate}
\end{lemma}

\begin{lemma} \label{reg} \quad
\begin{enumerate}
\item[(i)] $\reg N  \le \max \{\reg M, \reg P\},$
\item[(ii)] $\reg M \le \max \{\reg N, \reg P +1\},$
\item[(iii)] $\reg P \le \max \{\reg M - 1, \reg N\},$
\item[(iv)] $\reg M = \reg P+1$ if  $\reg N < \reg P$,
\item[(v)] $\reg P = \reg M-1$ if  $\reg N < \reg M$,
\item[(vi)] $\reg P = \reg N$ if  $\reg N > \reg M$,
\item[(vii)] $\reg N = \reg M$ if  $\reg P +1 < \reg M$. 
\end{enumerate}
\end{lemma}

These basic lemmas will be used in this paper without references.
Using the fact that a polynomial ring over a field is Cohen-Macaulay, and the characterization of regularity in terms of the minimal free resolution, one can deduce the following well-known lemma, which will also be used in this paper without references.

\begin{lemma}
Let $A$ be a polynomial ring over a field and let $I$ be a non-zero proper homogeneous ideal in $A$. Then \par
{\rm (i)} $\depth I = \depth A/I + 1$, \par
{\rm (ii)} $\reg I = \reg A/I +1$.
\end{lemma}

Let us conclude this section by fixing some notations for the remaining of the paper. Let $k$ be a field, and let $A = k[x_1,\dots,x_r]$ and $B = k[y_1,\dots,y_s]$ be polynomial rings over $k$.
Let $I \subset A$ and $J \subset B$ be arbitrary nonzero proper homogeneous ideals.
For convenience, we also use $I$ and $J$ to denote the extensions of $I$ and $J$ in
the polynomial ring $R = k[x_1,\dots,x_r,y_1,\dots,y_s]$. The main objectives of our investigation are the depth and the regularity of powers of the sum $I+J$ in $R$.


\section{Depth and regularity of $R/(I+J)^n$} \label{sec.regular}

The aim of this section is to give bounds for the depth and the regularity of $(I+J)^n$ in terms of
those of powers of $I$ and $J$. We shall start with the following observation.

\begin{lemma} \label{L1} \cite[Lemmas 1.1]{HT}
$IJ = I \cap J$.
\end{lemma}

Lemma \ref{L1} allows us to show that $(I+J)^n$ possesses an ascending sequence of subideals whose quotients of successive elements have a nice representation.

\begin{lemma} \label{intersection}
Let $Q_i = \sum_{j=0}^i I^{n-j}J^j$, $i = 0,\dots,n$. For $i \ge 1$ we have
$$Q_i/Q_{i-1} \cong  I^{n-i}J^i/I^{n-i+1}J^i.$$
\end{lemma}

\begin{proof}
It can be seen that
$$Q_i/Q_{i-1} = I^{n-i}J^i + Q_{i-1}/Q_{i-1} \cong   I^{n-i}J^i/Q_{i-1}\cap  I^{n-i}J^i.$$
Thus, it suffices to show that
$Q_{i-1}\cap  I^{n-i}J^i = I^{n-i+1}J^i.$\par

Note that
$Q_{i-1} =  I^n +  I^{n-1}J + \cdots +  I^{n-i+1}J^{i-1}.$
Then $Q_{i-1}\cap  I^{n-i}J^i \supseteq I^{n-i+1}J^i$.
On the other hand, since $Q_{i-1} \subseteq I^{n-i+1}$, we have
$$Q_{i-1}\cap  I^{n-i}J^i  \subseteq I^{n-i+1} \cap I^{n-i}\cap J^i
= I^{n-i+1}\cap J^i = I^{n-i+1}J^i,$$
where the last equality follows from Lemma \ref{L1}.
Therefore, $Q_{i-1}\cap  I^{n-i}J^i = I^{n-i+1}J^i$ and we are done.
\end{proof}

Thanks to Lemma \ref{intersection}, the depth and regularity of $Q_i/Q_{i-1}$ can be estimated by invoking the following results of Hoa and Tam \cite{HT}.

\begin{lemma} \label{L2} \cite[Lemma 2.2 and Lemma 3.2]{HT} ~\par
{\rm (i)} $\depth R/IJ = \depth A/I + \depth B/J + 1$.\par
{\rm (ii)} $\reg R/IJ = \reg A/I + \reg B/J + 1$.
\end{lemma}

We are now ready to give bounds for the depth and the regularity of $R/(I+J)^n$ in terms of those of powers of $I$ and $J$ in $A$ and $B$, respectively.

\begin{theorem}\label{first bound}
For all $n \ge 1$, we have \par
{\rm (i)} $\depth R\big/(I+J)^n \ge$\par
\centerline{\hspace*{5ex}$\displaystyle \min_{i \in [1,n-1], \ j \in [1,n]} \{\depth A/I^{n-i} + \depth B/J^i + 1, \depth A/I^{n-j+1} + \depth B/J^j\},$} \par
{\rm (ii)} $\reg R\big/(I+J)^n \le$\par
\centerline{$\displaystyle\max_{i \in [1,n-1], \ j \in [1,n]} \{\reg A/I^{n-i} + \reg B/J^i + 1, \reg A/I^{n-j+1} + \reg B/J^j \}.$}
\end{theorem}

\begin{proof}
Let $Q_i = \sum_{j=0}^i I^{n-j}J^j$, for $i = 0,\dots,n$. Note that $Q_n = (I+J)^n$.
Using the short exact sequence
$$0 \to Q_i/Q_{i-1} \to R/Q_{i-1} \to R/Q_i \to 0$$
we can deduce that
\begin{align}
\depth R/Q_n & \ge \min\{\depth R/Q_0, \depth Q_i/Q_{i-1}-1|\ i = 1,\dots,n\},\\
\reg R/Q_n & \le \max\{\reg R/Q_0, \reg Q_i/Q_{i-1}-1|\ i = 1,\dots,n\}.
\end{align}\par
Since $Q_0 = R/I^n$, we have
\begin{align*}
\depth R/Q_0 & = \depth A/I^n + s \ge \depth A/I^n + \depth B/J,\\
\reg R/Q_0 & = \reg A/I^n \le \reg A/I^n + \reg B/J.
\end{align*}
By Lemma \ref{intersection}, $Q_i/Q_{i-1} = I^{n-i}J^i/I^{n-i+1}J^i.$
Hence, we have the short exact sequence
$$0 \to  Q_i/Q_{i-1} \to R/I^{n-i+1}J^i \to R/I^{n-i}J^i \to 0.$$
It then follows that
\begin{align}
\depth Q_i/Q_{i-1} & \ge \min\{\depth R/I^{n-i}J^i+1, \depth R/I^{n-i+1}J^i\}, \label{eq.qii} \\
\reg Q_i/Q_{i-1} & \le \max\{\reg R/I^{n-i}J^i+1, \reg R/I^{n-i+1}J^i\}. \label{eq.qiii}
\end{align}

For $i = 1,\dots,n-1$, by applying Lemma \ref{L2} to compute the depth and regularity in the left hand side of (\ref{eq.qii}) and (\ref{eq.qiii}), we get
\begin{align*}
\depth Q_i/Q_{i-1}\! -\! 1 & \ge \min\{\depth A/I^{n-i}\! + \depth B/J^i + 1, \depth A/I^{n-i+1}\! + \depth B/J^i\},\\
\reg Q_i/Q_{i-1}\! -\! 1 & \le \max\{\reg A/I^{n-i} + \reg B/J^i + 1, \reg A/I^{n-i+1} + \reg B/J^i \}.
\end{align*}

For $i = n$, notice that
\begin{align*}
\depth R/J^n & = r+\depth B/J^n \ge \depth A/I + \depth B/J^n +1 = \depth R/IJ^n,\\
\reg R/J^n & = \reg B/J^n \le \reg A/I + \reg B/J^n + 1 = \reg R/IJ^n.
\end{align*}
Thus,
\begin{align*}
\depth Q_n/Q_{n-1}\! -\! 1 & \ge  \depth A/I + \depth B/J^n +1,\\
\reg Q_n/Q_{n-1}\! -\! 1 & \le  \reg A/I + \reg B/J^n + 1.
\end{align*} \par
The conclusion now follows by combining the above estimates for $\depth R/Q_0$, $\reg R/Q_0$, $\depth Q_i/Q_{i-1}-1$, and $\reg Q_i/Q_{i-1}-1$, for $i = 1,\dots,n$, together with (2.1) and (2.2).
\end{proof}

The bounds in Theorem \ref{first bound} are given by the minimum and maximum values of two terms.
The following propositions show that the values of both terms are attainable and, thus, both terms are essential in the statement of Theorem \ref{first bound}. For that we shall need the following result on tensor products of modules over a field.

\begin{lemma} \label{tensor}
Let $M$ and $N$ be graded module over $A$ and $B$, respectively. Then \par
{\rm (i)} $\depth M\otimes_k N  = \depth M + \depth N,$\par
{\rm (ii)} $\reg M\otimes_k N = \reg M + \reg N.$
\end{lemma}

\begin{proof}
Let $\MM$ denote the maximal ideal of $R$.
Note that
\begin{align*}
\depth M\otimes_k N & = \min\{k|\ H_\MM^k(M\otimes_k N) \neq 0\},\\
\reg M\otimes_k N = & \max\{n+k|\  H_\MM^k(M\otimes_k N)_n \neq 0, k \ge 0\}.
\end{align*}
The assertions follow from the following formula for the local cohomology modules of tensor products of Goto-Watanabe \cite[Theorem 2.2.5]{GW}:
$$H_\MM^k(M\otimes_k N) = \bigoplus_{i+j=k}H_\mm^i(M)\otimes_k H_\nn^j(N),$$
where $\mm$ and $\nn$ denote the maximal ideals of $A$ and $B$.
\end{proof}

\begin{proposition} \label{example 1}
Assume that $\depth A/I^2 \ge \depth A/I+1$. \par
{\rm (i)} If $\depth B/J^2 \ge \depth B/J+1$ then
$\depth R/(I+J)^2 = \depth A/I + \depth B/J + 1.$\par
{\rm (ii)} If $\depth B/J^2 < \depth B/J$ then
$\depth R/(I+J)^2 = \depth A/I + \depth B/J^2.$
\end{proposition}

\begin{proof}
From the proof of Theorem \ref{first bound}, we have the following short exact sequences
\begin{align}
& 0 \to IJ/I^2J  \to R/I^2 \to R/(I^2+IJ) \to 0,\\
& 0 \to J^2/IJ^2 \to R/(I^2+IJ) \to R/(I+J)^2 \to 0.
\end{align}
Note that $IJ = I \otimes_kJ$ for arbitrary ideals $I \subseteq A$ and $J \subseteq B$.
Thus, $IJ/I^2J \cong (I/I^2) \otimes_k J$ and $J^2/IJ^2 \cong (A/I) \otimes_k J^2$.
Therefore, it follows from Lemma \ref{tensor} that
\begin{align*}
\depth IJ/I^2J & = \depth I/I^2 + \depth J = \depth I/I^2 + \depth B/J +1,\\
\depth J^2/IJ^2 & = \depth A/I + \depth J^2 = \depth A/I + \depth B/J^2 +1.
\end{align*}

By considering the short exact sequence $0 \to I/I^2 \to A/I^2 \to A/I \to 0$, it follows from our hypotheses that
$\depth I/I^2 = \depth A/I+1$. Since  $\depth B/J+1 \le s$, we get
$$\depth R/I^2 = \depth A/I^2+s \ge \depth I/I^2 + \depth B/J +1 = \depth IJ/I^2J.$$
By (2.5), this implies that
$$\depth R/(I^2+IJ) \ge \depth IJ/I^2J - 1 = \depth A/I + \depth B/J+1.$$

(i) If $\depth B/J^2 \ge \depth B/J + 1$ then $\depth B/J + 1 \le \dim B/J < s$. Hence,
$\depth R/I^2 > \depth IJ/I^2J.$ We thus have
\begin{align*}
\depth R/(I^2+IJ) & = \depth IJ/I^2J - 1 = \depth A/I + \depth B/J+1\\
& \le \depth A/I + \depth B/J^2 = \depth J^2/IJ^2-1.
\end{align*}
By (2.6), this implies that
$$\depth R/(I+J)^2  = \depth R/(I^2+IJ) = \depth A/I + \depth B/J + 1.$$

(ii) If $\depth B/J^2 < \depth B/J$ then
\begin{align*}
\depth R/(I^2+IJ) & \ge \depth A/I + \depth B/J+1\\
& >  \depth A/I + \depth B/J^2+1 = \depth J^2/IJ^2.
\end{align*}
By (2.6), this implies that
$$\depth R/(I+J)^2 = \depth J^2/IJ^2 - 1 = \depth A/I + \depth B/J^2.$$
\end{proof}

To find ideals $I$ which satisfy the conditions of  Proposition \ref{example 1} we refer the reader to \cite[Theorem 4.1]{HH}, where for any bounded non-decreasing function $f\!: \NN \to \NN$, one constructs a monomial ideal $I$ such that $\depth A/I^n = f(n)$ for all $n \ge 1$.

\begin{proposition} \label{example 2}
Assume that $\reg A/I^2 \le \reg A/I+1$.  \par
{\rm (i)} If $\reg B/J^2 \le \reg B/I+1$ then
$\reg R/(I+J)^2 = \reg A/I+\reg B/J +1.$\par
{\rm (ii)} If $\reg B/J^2 > \reg B/I$ then
$\reg R/(I+J)^2 = \reg A/I+\reg B/J^2.$
\end{proposition}

\begin{proof}
The proof is similar to that of Proposition \ref{example 1}.
\end{proof}

It is easy to find ideals $I$ with $\reg A/I^2 > \reg A/I$.
However, to find ideals $I$ with $\reg A/I^2 \le \reg A/I +1$ is hard. 
We are thankful to A.~Conca, who communicated to us the following example.

\begin{example}
Let $A = k[x_1,x_2,x_3]$ and $I=(x_1^4,x_1^3x_2,x_1x_2^3, x_2^4,x_1^2x_2^2x_3^5)$.
Using Macaulay2, we get $\reg A/I = 8$ and $\reg A/I^2 = 7$.
\end{example}

If $J$ is generated by linear forms, we have the following formulae.
Note that in this case, $\depth B/J^n = n-1$ for $n\ge 1$.

\begin{proposition} \label{variable}
Assume that $J$ is generated by linear forms.
Then~\par
{\rm (i)} $\depth R/(I+J)^n = \min_{i \le n} \depth A/I^i + \dim B/J,$ \par
{\rm (ii)} $\reg R/(I+J)^n = \max_{i \le n}\{ \reg A/I^i-i\}+n.$
\end{proposition}

\begin{proof}
Without restriction we may assume that $J = (y_1,...,y_t)$, $t \le s$. Then $\dim B/J = s-t$.
Set $B' = k[y_1,...,y_t]$, $J' = (y_1,...,y_t)B'$, and $R' = k[x_1,...,x_r,y_1,...,y_t]$. Then
\begin{align*} 
\depth R/(I+J)^n & =  \depth R'/(I+J')^n+ s-t,\\
 \reg R/(I+J)^n & = \reg R'/(I+J')^n.
\end{align*}
Therefore, we only need to prove the case $t=s$.  \par
If $t=s=1$, we set $y=y_1$. Then $B = k[y]$ and $J = (y)$.
Write $R = \oplus_{i \ge 0} Ay^i$ and
$$(I,y)^n = I^n \oplus I^{n-1}y \oplus \cdots \oplus Ay^n \oplus Ay^{n+1} \oplus \cdots.$$
Then $R/(I,y)^n = \oplus_{i \le n} (A/I^i)y^{n-i}$. From this it follows that  
\begin{align*} 
\depth R/(I,y)^n & =  \min_{i\le n}\depth A/I^i ,\\
 \reg R/(I,y)^n & = \max_{i \le n}\{\reg A/I^i +n-i\}.
\end{align*}\par
If $t = s > 1$, we set $A' = k[x_1,...,x_r,y_1,...,y_{s-1}]$ and $I' = (I,y_1,...,y_{s-1})A'$.
Using induction we may assume that 
\begin{align*} 
\depth A'/(I')^n & =  \min_{i\le n}\depth A/I^i ,\\
 \reg A'/(I')^n & = \max_{i \leq n}\{\reg A/I^i-i\}+n.
\end{align*}
Note that $I+J = (I',y_s)$. Then we have
\begin{align*} 
\depth R/(I+J)^n & =  \min_{i\le n}\depth A'/(I')^i  = \min_{i\le n} \depth A/I^i,\\
 \reg R/(I+J)^n & = \max_{i \leq n}\{\reg A'/(I')^i -i\}+n = \max_{i \le n}\{ \reg A/I^i-i\}+n.
\end{align*} 
\end{proof}

Propositions \ref{example 1}, \ref{example 2} and \ref{variable} show that there is no general formulae for $\depth R/(I+J)^n$ and $\reg R/(I+J)^n$, and that both terms in the statement of Theorem \ref{first bound} are essential.


\section{Depth and regularity of $(I+J)^n/(I+J)^{n+1}$}

In this section, we shall see that the depth and regularity of $(I+J)^n/(I+J)^{n+1}$ can be nicely related to those associated to $I$ and $J$. In particular, we shall provide bounds for the depth and regularity of $R/(I+J)^n$ in terms of those of $I^i/I^{i+1}$ and $J^j/J^{j+1}$. \par

We start by making the following observation, which is essential for our results in this section.

\begin{lemma} \label{summand}
$I^i/I^{i+1} \otimes_k J^j/J^{j+1} \cong I^iJ^j /(I^{i+1}J^j + I^{i+1}J^j).$
\end{lemma}

\begin{proof}
It is clear that
\begin{align*}
I^i/I^{i+1} \otimes_k J^j/J^{j+1} &
\cong (I^i \otimes_k J^j/ I^i \otimes_kJ^{j+1})/(I^{i+1} \otimes_k J^j/I^{i+1}\otimes_kJ^{j+1}).\\
& \cong (I^iJ^j/ I^iJ^{j+1})/(I^{i+1}J^j/I^{i+1}J^{j+1}).
\end{align*}
Using Lemma \ref{L1}, we have
\begin{align*}
I^{i+1}J^j/I^{i+1}J^{j+1} & = I^{i+1}\cap J^j/I^{i+1} \cap J^{j+1} =  I^{i+1}\cap J^j/(I^i \cap J^{j+1}) \cap (I^{i+1}\cap J^i)\\
& \cong I^{i+1}\cap J^j + I^i \cap J^{j+1}/I^i \cap J^{j+1} = I^{i+1}J^j + I^iJ^{j+1}/I^iJ^{j+1}.
\end{align*}
Therefore,
\begin{align*}
I^i/I^{i+1} \otimes_k J^j/J^{j+1} & \cong (I^iJ^j/ I^iJ^{j+1})/(I^{i+1}J^j + I^iJ^{j+1}/I^iJ^{j+1})\\
& \cong I^iJ^j /(I^{i+1}J^j + I^{i+1}J^j).
\end{align*}
\end{proof}

\begin{proposition} \label{decomposition}
$\displaystyle (I+J)^n/(I+J)^{n+1}
= \bigoplus_{i+j=n}\big(I^i/I^{i+1} \otimes_k J^j/J^{j+1}\big).$
\end{proposition}

\begin{proof}
Note that $(I+J)^n = \sum_{i+j =n}I^iJ^j$.
Let $(i,j)$ and $(h,t)$ be two different pairs of non-negative integers such that $i+j = h+t = n$.
By Lemma \ref{L1}, we have
\begin{align*}
I^iJ^j \cap I^hJ^t & = (I^i \cap J^j) \cap (I^h \cap J^t) = I^{\max\{i,h\}} \cap J^{\max\{j,t\}}\\
& = J^{\max\{i,h\}}J^{\max\{j,t\}} \subseteq (I+J)^{n+1}.
\end{align*}
It then follows that
\begin{align}
(I+J)^n/(I+J)^{n+1} &= \bigoplus_{i+j=n} \big(I^iJ^j + (I+J)^{n+1}/(I+J)^{n+1}\big) \nonumber\\
& \cong  \bigoplus_{i+j=n} \big(I^iJ^j / (I+J)^{n+1}\cap I^iJ^j\big). \label{eq.decomp}
\end{align}
\par
It is easy to verify that $(I+J)^{n+1} \cap  I^i J^j \supseteq I^{i+1}J^j + I^{i+1}J^j.$
On the other hand, it can be seen that
$(I+J)^{n+1} \subseteq  I^{i+1} + J^{j+1}$. Hence,
\begin{align*}
(I+J)^{n+1} \cap I^iJ^j & \subseteq (I^{i+1} + J^{j+1}) \cap I^i \cap J^j \subseteq \big(I^{i+1}+ I^i \cap J^{j+1}\big) \cap J^j\\
& =  I^{i+1} \cap J^j + I^{i+1} \cap J^j
= I^{i+1}J^j + I^{i+1}J^j,
\end{align*}
where the last equality follows by applying Lemma \ref{L1}. Therefore,
$$(I+J)^{n+1} \cap  I^iJ^j = I^{i+1}J^j + I^{i+1}J^j.$$
By Lemma \ref{summand}, $I^iJ^j /\big(I^{i+1}J^j + I^{i+1}J^j\big) \cong I^i/I^{i+1} \otimes_k J^j/J^{j+1}.$
Thus,
$$I^iJ^j / (I+J)^{n+1}\cap I^iJ^j \cong I^i/I^{i+1} \otimes_k J^j/J^{j+1}.$$
The conclusion now follows from (\ref{eq.decomp}).
\end{proof}

The decomposition of $(I+J)^n/(I+J)^{n+1}$ in Proposition \ref{decomposition} yields the following formulae for
its depth and the regularity.

\begin{theorem} \label{equality}
For all  $n \ge 1$, we have\smallskip\par
{\rm (i)} $\displaystyle \depth (I+J)^n/(I+J)^{n+1} = \min_{i+j=n}\{\depth I^i/I^{i+1} + \depth J^j/J^{j+1}\},$\par
{\rm (ii)} $\displaystyle \reg (I+J)^n/(I+J)^{n+1} = \max_{i+j=n}\{\reg  I^i/I^{i+1} + \reg J^j/J^{j+1}\}.$
\end{theorem}

\begin{proof}
By Proposition \ref{decomposition}, we have
\begin{align*}
\depth (I+J)^n/(I+J)^{n+1} & = \min_{i+j=n} \depth (I^i/I^{i+1}\otimes_k J^j/J^{j+1}),\\
\reg (I+J)^n/(I+J)^{n+1} & = \max_{i+j=n} \reg (I^i/I^{i+1}\otimes_k J^j/J^{j+1}).
\end{align*}
Hence, the assertions follow by applying Lemma \ref{tensor}.
\end{proof}

Now, we can relate the depth and regularity of $R/(I+J)^n$ to those of successive quotients associated to $I$ and $J$.

\begin{corollary} \label{second bound}
For all  $n \ge 1$, we have\smallskip\par
{\rm (i)} $\displaystyle \depth R/(I+J)^n \ge  \min_{i+j \le n-1}\{\depth I^i/I^{i+1} + \depth J^j/J^{j+1}\},$\par
{\rm (ii)} $\displaystyle \reg R/(I+J)^n \le \max_{i+j\le n-1}\{\reg  I^i/I^{i+1} + \reg J^j/J^{j+1}\}.$
\end{corollary}

\begin{proof}
Using the short exact sequences
$$0 \to (I+J)^t/(I+J)^{t+1} \to R/(I+J)^{t+1} \to R/(I+J)^t \to 0$$
for $t < n$, we deduce that
\begin{align*}
\depth R/(I+J)^n & \ge  \min_{t \le n-1} \depth (I+J)^t/(I+J)^{t+1},\\
\reg R/(I+J)^n & \le \max_{t \le n-1} \reg (I+J)^t/(I+J)^{t+1}.
\end{align*}
Hence, the assertions follow from Theorem \ref{equality}.
\end{proof}

\begin{remark}
For an arbitrary ideal $I \subseteq A$, by considering the short exact sequence
$$0 \longrightarrow I^i/I^{i+1} \longrightarrow A/I^{i+1} \longrightarrow A/I^i \longrightarrow 0,$$
it can be seen that
\begin{align*}
\depth I^i/I^{i+1} & \ge \min\{\depth A/I^i+1, \depth A/I^{i+1}\},\\
\reg I^i/I^{i+1} & \le \max\{\reg A/I^i+1, \reg A/I^{i+1}\}.
\end{align*}
Thus, from Corollary \ref{second bound}, one can derive bounds for $\depth R/(I+J)^n$ and $\reg R/(I+J)^n$ in terms of $\depth A/I^i, \depth B/J^j$, $\reg A/I^i$, and $\reg B/J^j$. However, these bounds are worse than what was given in Theorem \ref{first bound}, since they involve all indices $i, j$ with $i+j \le n$.
\end{remark}

It often happens that  $\depth I^{i-1}/I^i  \ge \depth I^i/I^{i+1}$ and, if $I$ is generated by forms of degree $\ge 2$,
$\reg I^{i-1}/I^i +2 \le \reg I^i /I^{i+1}$. In these situations, we have the following formulae for $\depth R/(I+J)^n$ and $\reg R/(I+J)^n$.

\begin{corollary} ~\par
{\rm (i)} If  $\depth I^{i-1}/I^i  \ge \depth I^i/I^{i+1}$ for $i \le n-1$, then
$$\depth R/(I+J)^n = \min_{i+j = n-1}\{\depth I^i /I^{i+1}+ \depth J^j/J^{i+1}\}.$$\par
{\rm (ii)} If $\reg I^{i-1}/I^i +2 \le \reg I^i /I^{i+1}$ for $i \le n-1$, then
$$\reg R/(I+J)^n = \max_{i+j = n-1}\{\reg I^i /I^{i+1} + \reg J^j/J^{i+1}\}.$$
\end{corollary}

\begin{proof}
(i) By Theorem \ref{equality}.(i) we only need to prove that
$$\depth R/(I+J)^n = \depth (I+J)^{n-1}/(I+J)^n.$$
For $n =1$, this is trivial. For $n > 1$,  we may assume that $\depth R/(I+J)^{n-1} = \depth (I+J)^{n-2}/(I+J)^{n-1}$. By Theorem \ref{equality}.(i) and the assumption $\depth I^{i-1}/I^i  \ge \depth I^i/I^{i+1}$, $i \le n-1$, we can see that $\depth (I+J)^{n-2}/(I+J)^{n-1} \ge \depth (I+J)^{n-1}/(I+J)^n$.
Thus, $\depth R/(I+J)^{n-1} \ge \depth (I+J)^{n-1}/(I+J)^n$. Now, from the exact sequence
$$0 \to (I+J)^{n-1}/(I+J)^n \to R/(I+J)^n \to R/(I+J)^{n-1} \to 0$$
we get $\depth R/(I+J)^n = \depth (I+J)^{n-1}/(I+J)^n$. \par

(ii) As in the proof of (i) we only need to prove that
$$\reg R/(I+J)^n = \reg (I+J)^{n-1}/(I+J)^n$$
for $n > 1$. By induction we may assume that $\reg R/(I+J)^{n-1} = \reg (I+J)^{n-2}/(I+J)^{n-1}.$
From Theorem \ref{equality}.(ii) and the assumption $\reg I^{i-1}/I^i +2 \le \reg I^i /I^{i+1}$, $i \le n-1$, 
we can deduce that $\reg (I+J)^{n-2}/(I+J)^{n-1}+ 2 \le \reg (I+J)^{n-1}/(I+J)^n$. Thus,
$\reg R/(I+J)^{n-1}+2 \le \reg (I+J)^{n-1}/(I+J)^n$. Hence, from the above exact sequence
we get $\reg R/(I+J)^n = \reg (I+J)^{n-1}/(I+J)^n$.
\end{proof}

Theorem \ref{equality}.(i)  has a nice application on the Cohen-Macaulayness of $R/(I+J)^n$.

\begin{proposition}  \label{CM}
$R/(I+J)^i$ is Cohen-Macaulay for all $i \le n$ if and only if $A/I^i$ and $B/J^i$ are Cohen-Macaulay for all $i \le n$.
\end{proposition}

\begin{proof}
It is obvious that
$$\dim (I+J)^{n-1}/(I+J)^n\! = \dim R/(I+J)\! = \dim A/I+ \dim B/J\! = \dim I^i/I^{i+1} + \dim J^j/J^{j+1}.$$
By Corollary \ref{second bound}.(i), it can be seen that $\depth (I+J)^{n-1}/(I+J)^n = \dim R/(I+J)$ if and only if $\depth I^{i-1}/I^i = \dim A/I$ and $\depth J^{i-1}/J^i = \dim B/J$ for all $i \le n$. It follows that $(I+J)^{n-1}/(I+J)^n$ is Cohen-Macaulay if and only if $I^{i-1}/I^i$ and $J^{i-1}/J^i$ are Cohen-Macaulay for all $i \le n$. In particular, these conditions imply that $(I+J)^{i-1}/(I+J)^i$ are Cohen-Macaulay for $i \le n$.
\par
On the other hand, using induction and the exact sequence
$$0 \to (I+J)^{i-1}/(I+J)^i \to R/(I+J)^i \to R/(I+J)^{i-1} \to 0$$
we can prove that $(I+J)^{i-1}/(I+J)^i$ is Cohen-Macaulay for $i \le n$ if and only if $R/(I+J)^i$ is Cohen-Macaulay for $i \le n$. Similarly, $I^{i-1}/I^i$ and $J^{i-1}/J^i$ are Cohen-Macaulay for $i \le n$ if and only if $A/I^i$ and $B/J^i$ are Cohen-Macaulay for $i \le n$. The result is proved.
\end{proof}

The proof of Theorem \ref{CM} exhibits an interesting phenomenon that if $(I+J)^{n-1}/(I+J)^n$ is Cohen-Macaulay, then so is $R/(I+J)^i$ for $i \le n$. This implication does not hold in general for an arbitrary ideal. 

\begin{example}
{\rm Let $I = A = K[x,y,z]$ and $I = (x^4, x^3y, xy^3, y^4, x^2y^2z)$. Then $\dim A = 1$, $\depth A/I = 0$, and $\depth A/I^2 = 1$ (see \cite[Theorem 4.1]{HH}). From the exact sequence $0 \to I/I^2 \to A/I^2 \to A/I \to 0$ one can deduce that $\depth I/I^2 = 1$. Hence,  $I/I^2$ is Cohen-Macaulay but $A/I$ is not.}
\end{example}


\section{Asymptotic behavior of depth}

For an arbitrary ideal $I$, it is a well celebrated result of Brodman \cite{Br} that $\depth(A/I^n)$ is a constant function for $n \gg 0$.  There has been  much interest in estimating the stable value $\lim_{n\to \infty}\depth A/I^n$ and the stability index of the function $\depth A/I^n$ \cite{HH,HV,NV,Tr}. In this section, we describe the stable value of $\depth R/(I+J)^n$ in terms of those of $\depth A/I^n$ and $\depth B/J^n$. \par

We first recall the following result of Herzog and Hibi.

\begin{lemma} \label{HH} \cite[Theorem 1.2]{HH}
For an arbitrary homogeneous ideal $I$,  $\depth I^{i-1}/I^i$ is a constant for $i \gg 1$ and
$$\lim_{i\to \infty}\depth A/I^i = \lim_{i \to \infty} \depth I^{i-1}/I^i.$$
\end{lemma}

Let us denote by $s(I)$ the stability index of the function $\depth I^{i-1}/I^i$, that is, $s(I)$ is the least integer $m$ such that $\depth I^{i-1}/I^i = \depth I^i/I^{i+1}$ for $i \ge m$. Note that
$$\min_{i \ge 1} \depth I^{i-1}/I^i = \min_{i \le s(I)} \depth I^{i-1}/I^i.$$

For large enough $n$, we obtain the following formula for $\depth (I+J)^{n-1}/(I+J)^n$.

\begin{proposition} \label{d-asymptotic}
For $n \ge s(I)+s(J)-1$, we have \smallskip\par
$\depth (I+J)^{n-1}/(I+J)^n = $\smallskip\par
\centerline{$\displaystyle \min\big\{\lim_{i \to \infty}\depth I^{i-1}/I^i + \min_{j \le s(J)} \depth J^{j-1}/J^j, \min_{i \le s(I)}\depth I^{i-1}/I^i +  \lim_{j \to \infty}\depth J^{j-1}/J^j\big\}.$}
\end{proposition}

\begin{proof}
It follows from Theorem \ref{equality}.(i) that
$$\depth (I+J)^{n-1}/(I+J)^n = \min_{i+j=n+1} \{\depth I^{i-1}/I^i + \depth J^{j-1}/J^j\}.$$
By the definition of $s(I)$ and $s(J)$, we have that
$\depth I^{i-1}/I^i = \lim_{i \to \infty}\depth I^{i-1}/I^i$ for $i \ge s(I)$ and $\depth J^{i-1}/J^j = \lim_{j \to \infty}\depth J^{j-1}/J^j$ for $j \ge s(J)$.

Consider $n \ge s(I) + s(J)-1$ and $i+j = n+1$. Then, $i \ge s(I)$ if $j \le s(J)$. Therefore,
$$\depth I^{i-1}/I^i + \depth J^{i-1}/J^j =
\begin{cases} \displaystyle \lim_{i \to \infty}\depth I^{i-1}/I^i  + \depth J^{i-1}/J^j\  \text{ if }\ j \le  s(J),\\
\displaystyle \depth I^{i-1}/I^i +  \lim_{j \to \infty}\depth J^{j-1}/J^j\ \text{ if }\ j \ge s(J). \end{cases}$$
Note that $i \le n-s(J)+1$ if and only if $j \ge s(J)$, and that
$$\min_{i \le n-s(J)+1} \depth I^{i-1}/I^i = \min_{i \le s(I)} \depth I^{i-1}/I^i$$
since $n-s(J)+1\ge s(I)$. Then we have \smallskip \par
$\displaystyle \min_{i+j=n+1} \{\depth I^{i-1}/I^i + \depth J^{j-1}/J^j\} =$\smallskip \par
\centerline{$\displaystyle \min\big\{\lim_{i \to \infty}\depth I^{i-1}/I^i + \min_{j \le s(J)} \depth J^{j-1}/J^j, \min_{i \le s(I}\depth I^{i-1}/I^i +  \lim_{j \to \infty}\depth J^{j-1}/J^j\big\}$}\par
\noindent for $n \ge s(I)+s(J)-1$, which yields the assertion.
\end{proof}

Proposition \ref{d-asymptotic} immediately gives the following bound for the index of stability of the function $\depth (I+J)^{n-1}\big/(I+J)^n$.

\begin{corollary} \label{stability}
$s(I+J) \le s(I)+s(J)-1$.
\end{corollary}

To derive a formula for $\lim_{n\to \infty}\depth R/(I+J)^n$ in terms of those of $A/I^i$ and $B/J^j$, we shall need the following observation.

\begin{lemma} \label{min}
$\displaystyle \min_{i \ge 1} \depth A/I^i =  \min_{i \ge 1} \depth I^{i-1}/I^i$.
\end{lemma}

\begin{proof}
Let $m$ be the least integer such that
$$\depth I^{m-1}/I^m = \min_{i \ge 1}\depth I^{i-1}/I^i.$$
Using the short exact sequences $0 \to I^{j-1}/I^j \to A/I^j \to A/I^{j-1} \to 0$
for $j \le i$,  we get
$$\depth A/I^i \ge \min_{j \le i} \depth I^{j-1}/I^j \ge \depth I^{m-1}/I^m$$
for all $i \ge 1$. In particular, $\depth A/I^{m-1} \ge \depth I^{m-1}/I^m$.
Hence, from the short exact sequence
$0 \to I^{m-1}/I^m \to A/I^m \to A/I^{m-1} \to 0,$
we get $\depth A/I^m = \depth I^{m-1}/I^m$.
Thus, $\min_{i \ge 1} \depth A/I^i = \depth I^{m-1}/I^m.$
\end{proof}

\begin{remark}
In general, we do not have $\min_{i\ge 1}\depth A/I^i = \lim_{i\to \infty}\depth A/I^i$ (see, for example, \cite[Theorem 0.1]{BHH}).
\end{remark}

\begin{theorem} \label{lim}
$\displaystyle  \lim_{n\to \infty} \depth R/(I+J)^n =$\smallskip\par
\centerline{$\displaystyle  \min\big\{\lim_{i \to \infty}\depth A/I^i + \min_{j \ge 1}\depth B/J^j,\, \min_{i \ge 1} \depth A/I^i + \lim_{j \to \infty}\depth B/J^j\big\}.$}
\end{theorem}

\begin{proof}
It follows from Lemma \ref{HH} and Proposition \ref{d-asymptotic} that
$$\lim_{n \to \infty}\depth R/(I+J)^n  =  \lim_{n \to \infty} \depth (I+J)^{n-1}/(I+J)^n =$$
\centerline{$\displaystyle \min\big\{\lim_{i \to \infty}\depth I^{i-1}/I^i + \min_{j \ge 1} \depth J^{j-1}/J^j, \min_{i \ge 1}\depth I^{i-1}/I^i +  \lim_{j \to \infty}\depth J^{j-1}/J^j\big\}.$}\par
By Lemma \ref{HH} and Lemma \ref{min}, we can replace $I^{i-1}/I^i$, $J^{j-1}/J^j$ by $A/I^i$, $B/J^j$ to obtain the assertion.
\end{proof}

We are unable to deduce from Corollary \ref{stability} a bound for the stability index of the function $\depth R/(I+J)^n$.
In general, there seems to be no relationships between the stability indices of the functions $\depth A/I^n$ and $\depth I^{n-1}/I^n$. \par

We shall now give two interesting applications on the depth functions.
Recall that an ideal is said to \emph{have a constant depth function} if the depth of all its powers are the same.
This notion was introduced by Herzog and Vladiou in \cite{HV}. The following result was proved in \cite{HV} under the additional assumption that the Rees rings of $I$ and $J$ are Cohen-Macaulay.

\begin{proposition} \label{HV} {\rm (cf. \cite[Theorem 1.1]{HV})}
Let $I$ and $J$ be squarefree monomial ideals.
Then $I + J$ has a constant depth function if and only if so do $I$ and $J$.
\end{proposition}

\begin{proof}
Assume that $I+J$ has a constant depth function.
Let $d = \lim_{n\to \infty} \depth A/I^n$ and $e =  \lim_{n \to \infty}\depth B/J^n$.
By Theorem $\ref{lim}$, we have $\depth R/(I+J) \le d+e$.
By the proof of \cite[Theorem 2.6]{HTT}, $\depth A/I^n \le \depth A/I$ and $\depth B/J^n \le \depth B/J$ for all $n \ge 1$. As a consequence,
$$d + e \le \depth A/I + \depth B/J = \depth R/(I+J).$$
Therefore, $\depth R/(I+J) = d+e$, $\depth A/I = d$ and $\depth B/J = e$.
By Theorem $\ref{lim}$, this implies that $\depth A/I = \min_{n \ge 1} \depth A/I^n$ and
$\depth B/J = \min_{n \ge 1} \depth B/J^n$. Thus, we must have $\depth A/I^n = \depth A/I$ and $\depth B/J^n = \depth B/J$ for all $n \ge 1$.
\par

Conversely, assume that $I$ and $J$ have constant depth functions.
By Theorem \ref{first bound}, we obtain
$$\depth R/(I+J)^n   \ge  \depth A/I +\depth B/J = \depth R/(I+J)$$
for all $n \ge 1$.  Following the proof of \cite[Theorem 2.6]{HTT}, we also get that $\depth R/(I+J)^n  \le \depth R/(I+J)$. Therefore, $\depth R/(I+J)^n = \depth R/(I+J)$  for all $n \ge 1$.
\end{proof}

It is easy to construct examples showing that Theorem \ref{HV} does not hold for non-squarefree monomial ideals.

\begin{example}
Take $I$ to be a monomial ideal with $\depth I^n/I^{n+1} = 0$ for all $n \ge 1$ (e.g. $\dim A/I = 0$) and $J$ a monomial ideal with $\depth B/J = 0$ such that $J$ does not have a constant depth function (see e.g. \cite{HH}). Then
$\depth (I+J)^{n-1}/(I+J)^n = 0$ for all $n \ge 1$ by Theorem \ref{equality}. This implies that $\depth R/(I+J)^n = 0$ for all $n \ge 1$.
\end{example}

Given a convergent non-negative integer valued function $f$, Herzog and Hibi \cite{HH} conjectured that there exists a monomial ideal $Q$ in a polynomial ring $R$ such that $\depth R/Q^n = f(n)$ for all $n \ge 1$. They proved the conjecture for all bounded non-decreasing functions and a special class of non-increasing functions. We complete the statement for all non-increasing functions in the following theorem.

\begin{theorem} \label{non-increasing}
For any non-increasing function $f\!: \NN \to \NN$ there exist a monomial ideal $Q$ in a polynomial ring $R$ such that $\depth R/Q^n = f(n)$ for all $n \ge 1$.
\end{theorem}

\begin{proof}
We shall first show that there exists a monomial ideal $Q$ in a polynomial ring $R$ over $k$ such that $\depth Q^{n-1}/Q^n = f(n)$ for all $n \ge 1$.  Let $t$ be the least integer such that $f(n) = f(n+1)$ for all $n \ge t$. \par
Assume that $f(t) = 0$. Then $f(n) = 0$ for all $n \ge t$.
If $t = 0$ then choose $Q$ to be any monomial ideal with $\dim R/Q = 0$.
If $t > 0$ then consider the function $g(n) = f(n)-1$ for $n < t$, $g(n) = 0$ for $n \ge t$.
Let $v$ be the least integer such that $g(n) = g(n+1)$ for all $n \ge v$.
Using induction on $\sum_{n =1}^tf(n)$, we may assume that there exists a monomial ideal $I$ in a polynomial ring $A$ such that $\depth I^{n-1}/I^n = g(n)$ for all $n \ge 1$.
Choose $J$ to be a monomial ideal in a polynomial ring $B$ such that $\depth J^{n-1}/J^n = 1$ for $n < t-v$ and $\depth J^{n-1}/J^n = 0$ for $n > t-v$ (see Example \ref{1-0}). Set $Q = I+J$.
By Theorem \ref{equality}.(i), it can be easily verified that $\depth Q^{n-1}/Q^n = f(n)$.
\par
If $f(t) > 0$ then then $f(n) > 0$ for all $n \ge 1$. Consider $g(n) = f(n)-1$ for all $n \ge 1$.
Using induction on  $f(t)$, we may assume that there exists a monomial ideal $I$ in a polynomial ring $A$ such that $\depth I^{n-1}/I^n = g(n)$ for all $n \ge 1$. Choose $J$ to be a monomial ideal in a polynomial ring $B$ such that $\depth J^{n-1}/J^n = 1$ for all $n \ge 1$ (e.g. $J = (y_1) \subset k[y_1,y_2] = B$). Put $Q = I+J$.
Then Theorem \ref{equality}(i) yields $\depth Q^{n-1}/Q^n = f(n)$ for all $n \ge 1$. \par
Once there exists a homogeneous ideal $Q$ such that $\depth Q^{n-1}/Q^n = f(n)$ for all $n \ge 1$, we can proceed to show that $\depth R/Q^n = \depth Q^{n-1}/Q^n$ as follows. The assertion is trivial for $n=1$. For $n \ge 2$, we may assume that $$\depth R/Q^{n-1} = \depth Q^{n-2}/Q^{n-1} < \depth Q^{n-1}/Q^n.$$
Then, using the short exact sequence
$0 \to Q^{n-1}/Q^n \to R/Q^n \to R/Q^{n-1} \to 0$, we deduce that
$\depth R/Q^n = \depth Q^{n-1}/Q^n $.
\end{proof}

\begin{example} \label{1-0}
Let $I$ be the integral closure of the ideal
$(x_1^{3t}, x_1 x_2^{3t-2}x_3, x_2^{3t-1}x_3)^3$, for some $t \ge 1$, in the ring $A=k[x_1,x_2,x_3]$. Then, $\dim A/I=1$ and $I^n$ is integrally closed for $n\geq 1$.

By \cite[Proposition 4]{Tr1}, $\Ass I^{n-1}/I^n=\Ass A/I^n$, where $\Ass$ denotes the set of the associated primes. By \cite[Example, p.~54]{Tr1}, $(x_1,x_2,x_3) \in \Ass A/I^n$ if and only if $n \ge r$. Thus,
$\depth I^{n-1}/I^n = 1$ for $n < t$ and $\depth I^{n-1}/I^n = 0$ for $n \ge t$.

Note that examples with this property do not belong to the class of ideals with non-increasing depth functions constructed by Herzog and Hibi in \cite{HH}.
\end{example}


\section{Asymptotic behavior of regularity}

For an arbitrary homogeneous ideal $I$ in $A$, it is well known that $\reg I^n = dn +e$ for $n \gg 0$, where $d$ is the minimum of the maximal generating degree of a homogeneous reduction of $I$ and $e$ is a non-negative integer \cite{CHT, Ko}. Recall that a homogeneous ideal $Q \subseteq I$ is called a reduction of $I$ if $I^{n+1} = QI^n$ for some $n \ge 1$.
We denote by $\lin(I)$  the least natural number $m$ such that $\reg I^n = dn+e$ for $n \ge m$.
In general, it is hard to estimate $e$ and $\lin(I)$; see, for example, \cite{Be,Ch,EH,EU}. In this section, we investigate these invariants of $(I+J)$ in terms of those of $I$ and $J$. \par

We shall need the following lemmas on the relation between $\reg I^n$ and  $\reg I^{n-1}/I^n$.

\begin{lemma} \label{lin}
$\reg I^{n-1}/I^n = \reg I^n-1$ for $n \ge \lin(I)+1.$
\end{lemma}

\begin{proof}
We have $\reg  I^n = dn + e  > d(n-1)+e = \reg I^{n-1}$ for $n \ge  \lin(I)+1$. Therefore, using the short exact sequence
$0 \to I^n \to  I^{n-1} \to I^{n-1}/I^n \to 0,$
we deduce that $\reg I^{n-1}/I^n = \reg I^n-1$ for $n \ge \lin(I)+1$.
\end{proof}

In general, $\lin(I)$ is not the least integer $m$ such that $\reg I^{n-1}/I^n$ becomes a linear function for $n \ge m+1$.

\begin{example}
Let $A = k[x_1,x_2]$ and $I = (x_1^3,x_1^2x_2,x_1x_2^2,x_2^3)$. Then $I^n = (x_1,x_2)^n$ for $n \ge 2$.
Therefore, it is easy to check that $\reg I = 4$ and $\reg I^n = 3n$ for $n \ge 2$, while $\reg A/I = 3$ and $\reg I^{n-1}/I^n = 3n-1$ for $n \ge 2$.
\end{example}

\begin{lemma} \label{max} For any positive integers $c$ and $t$ we have
$$\max_{n \le t} \{\reg I^{n-1}/I^n - cn\} = \max_{n \le t}\{\reg I^n - cn\}-1.$$
\end{lemma}

\begin{proof}
Let $m \le t$ be the least positive integer such that
$$\reg I^{m-1}/I^m - cm = \max_{n \le t}\{\reg I^{n-1}/I^n - cn\}.$$
Using the short exact sequences $0 \to  I^{i-1}/I^i  \to A/I^i \to A/I^{i-1}\to 0$
for $i \le n$, we deduce that $\reg A/I^n \le  \max_{i \le n} \reg I^{i-1}/I^i.$
Thus,
$$\reg A/I^n - cn \le \max_{i \le n}\{ \reg I^{i-1}/I^i - ci\} \le \reg I^{m-1}/I^m - cm$$
for $n \le t$. In particular,
$$\reg A/I^{m-1} - c(m-1) \le \max_{i \le m-1}\{ \reg I^{i-1}/I^i - ci\} < \reg I^{m-1}/I^m - cm.$$
Therefore, $\reg A/I^{m-1} +1 < \reg I^{m-1}/I^{m}$.
Using the above short exact sequence for $i = m$, it follows that $\reg A/I^m = \reg I^{m-1}/I^m$. Hence, $\reg A/I^m - cm = \reg I^{m-1}/I^m - cm.$
This implies that
$$\max_{n \le t}\{\reg A/I^n - cn\} = \reg I^{m-1}/I^m - cm.$$
Note  that $\reg I^n = \reg A/I^n +1$. Then $\max_{n \le t}\{\reg I^n - cn\} = \reg I^{m-1}/I^m - cm+1$ and the conclusion follows.
\end{proof}

\begin{remark}
If $c = d$ and $t = \lin(I)$ we have  
$$\max_{n \le \lin(I)}\{\reg I^n - dn\} = \max_{n \ge 0}\{\reg I^n - cn\}.$$
If $\dim A/I = 0$ and $I$ is generated in a single degree, Eisenbud and Harris \cite[Proposition 1.1]{EH} proved that the function $\reg I^n - dn$ is non-decreasing  (see also \cite{Be}). In this case, if $\reg I^n - dn$ is not a constant function, $\reg I^n - dn$  has its maximum value at $n =1 < \lin(I)$. This shows that 
$\max_{n \le \lin(I)}\{\reg I^n - cn\}$ needs not be attained at $n = \lin(I)$.
\end{remark}

Using Theorem \ref{equality} we find the following asymptotic formula for $\reg (I+J)^{n-1}/(I+J)^n$.

\begin{proposition} \label{r-quotient}
Assume that $\reg I^n = dn + e$ and $\reg J^n = cn+f$ for $n \gg 0$, $c \ge d$.
Set $e^* = \max_{i \le \lin(I)}\{\reg I^i - ci\}$ and $f^* = \max_{j \le \lin(J)}\{ \reg J^j - dj\}$.
For $n \ge \lin(I)+\lin(J)$ we have
$$\reg (I+J)^{n-1}/(I+J)^n = \max\{d(n+1)+e+f^*, c(n+1) +f+e^*\}-2.$$
\end{proposition}

\begin{proof}
By Theorem \ref{equality}.(ii), we have
$$\reg (I+J)^{n-1}/(I+J)^n = \max_{i+j=n+1}\{\reg I^{i-1}/I^i + \reg J^{j-1}/J^j\}.$$
Consider $n \ge \lin(I) + \lin(J)$ and $i+j = n+1$. It can be seen that $i \ge \lin(I)+1$ if $j \le \lin(J)$. By Lemma \ref{lin}, $\reg I^{i-1}/I^i = di+e-1$ if $i \ge \lin(I)+1$ and $\reg J^{j-1}/J^j = cj+f-1$ if $j \ge \lin(J)+1$. Therefore,
\par
$\reg I^{i-1}/I^i + \reg J^{j-1}/J^j =
\begin{cases} di + e-1 + \reg J^{j-1}/J^j &  \text{if }\ j \le  \lin(J),\\
\reg I^{i-1}/I^i + cj + f-1 & \text{if }\  j \ge \lin(J)+1.
\end{cases}
$ \par
\centerline{$= \begin{cases}
d(n+1) + e-1 + \reg J^{j-1}/J^j-dj & \text{if }\  j  \le  \lin(J),\\
c(n+1) + f -1 +\reg I^{i-1}/I^i-ci & \text{if }\  j \ge \lin(J)+1.
\end{cases}$}\smallskip \par
Note that $j \ge \lin(J)+1$ if and only if  $i \le n-\lin(J)$. By Lemma \ref{max}, 
\begin{align*}
\max_{j \le \lin(J)} \{\reg J^{j-1}/J^j-dj\} & = \max_{j \le \lin(J)} \{\reg J^j-dj\}-1 = f^*-1,\\
\max_{i \leq n-\lin(J)}\{\reg I^{i-1}/I^i-ci\} & = \max_{i \leq n-\lin(J)} \{\reg I^i-ci\}-1.
\end{align*}
For $i \ge \lin(I)$, we have $\reg I^i - ci = (d-c)i + e$,  which is a non-increasing function because $c\ge d$. 
Since $n - \lin(J) \ge \lin(I)$, this property implies
$$\max_{i \leq n-\lin(J)} \{\reg I^i-ci\} = \max_{i \le \lin(I)} \{\reg I^i -ci\}= e^*.$$
Therefore,
$$\max_{i \leq n-\lin(J)}\{\reg I^{i-1}/I^i-ci\} = e^*-1.$$ \par
Taking into account the maximum in the two cases $j \le \lin(I)$ and $j \ge \lin(J)+1$, we get
\begin{align*}
\max_{i+j=n+1}\{\reg  I^{i-1}/I^i + \reg J^{j-1}/J^j\} &= \max\big\{d(n+1) +e + f^*-2,
c(n+1) + f+e^*-2\big\}\\
&= \max\big\{d(n+1) +e + f^*,c(n+1) + f+e^*\big\}-2
\end{align*}
for $n \ge \lin(I)+\lin(J)$.
\end{proof}

We are now ready to give the asymptotic linear function of $\reg (I+J)^n$.

\begin{theorem} \label{r-asymptotic}
Assume that $\reg I^n = dn + e$ and $\reg J^n = cn+f$ for $n \gg 0$, $c \ge d$.
Set $e^* = \max_{i \le \lin(I)}\{\reg I^i - ci\}$ and $f^* = \max_{j \le \lin(J)}\{ \reg J^j - dj\}$.
Then, for $n \gg 0$, we have
$$\reg (I+J)^n =
\begin{cases} c(n+1)  + f+e^*-1 & \text{if } c > d,\\
 d(n+1) + \displaystyle \max\{f+e^*, e+f^*\}-1 & \text{if } c = d.
 \end{cases}$$
\end{theorem}

\begin{proof}
By Lemma \ref{lin}, we have $\reg (I+J)^n = \reg (I+J)^{n-1}/(I+J)^n+1$ for $n \gg 0$.
Hence, the conclusion follows by applying Proposition \ref{r-quotient} for $n \gg 0$.
\end{proof}

We can also give an upper bound for $\lin(I+J)$ in terms of related invariants of $I$ and $J$. This follows from the following formula for $\reg (I+J)^n$.

\begin{proposition} \label{r-formula}
Assume that $\reg I^n = dn + e$ and $\reg J^n = cn+f$ for $n \gg 0$, $c \ge d$.
Set $e^* = \max_{i \le \lin(I)}\{\reg I^i - ci\}$ and $f^* = \max_{j \le \lin(J)}\{ \reg J^j - dj\}$.
For $n \ge \lin(I)+\lin(J)+1,$ we have
$$\reg (I+J)^n = \max\{d(n+1) +e+f^*, c(n+1)+f +e^*\}-1.$$
\end{proposition}

\begin{proof}
If $d = 1$, then $I$ has a reduction $Q$ generated by linear forms.
Since $Q$ is a prime ideal, $\sqrt{I} = Q \subseteq I$. This implies $I= Q$.
Hence, $\reg I^n = n$ for all $n \ge 1$. Thus, $\lin(I) = 1$, $e = 0$ and $e^* = 1-c$.
By Proposition \ref{variable},
$$\reg (I+J)^n = \max_{i \le n} \{\reg J^i  - i\}+n$$
for $n \ge 1$. For $\lin(J) \le i \le n$, $\reg J^i - i +n= ci + f - i +n = cn+ f + (1-c)(n-i)$. Hence
$\max_{\lin(J) \le i \le n} \{\reg J^i  - i\}+n = cn+f.$
From this it follows that
$$\max_{i \le n} \{\reg J^i  -i\}+n = \max\{f^*+n\!, cn+f\} = \max\{n+1+e+f^*\!, c(n+1)+f+e^*\}-1.$$
for $n \ge \lin(J)$. 
\par
If $d \ge 2$, we will show below that $\reg R/(I+J)^n = \reg (I+J)^{n-1}/(I+J)^n$
for $n \ge \lin(I)+\lin(J)+1$. By Proposition \ref{r-quotient}, this will imply the assertion.
\par
By Theorem \ref{first bound}, we have \smallskip\par
$\reg R/(I+J)^{n-1} \le$\par
\centerline{$\displaystyle\max_{i \in [1,n-2], \ j \in [1,n-1]} \{\reg A/I^{n-i-1} + \reg B/J^i + 1, \reg A/I^{n-j} + \reg B/J^j \}$ =}\par
\centerline{$\displaystyle\max_{i \in [1,n-2], \ j \in [1,n-1]} \{\reg I^{n-i-1} + \reg J^i -1, \reg I^{n-j} + \reg J^j -2\}.$}\par
\noindent Similar to the proof of Proposition \ref{r-quotient}, we can show that
$$\max_{j \in [1,n-1]}\{\reg I^{n-j} + \reg J^j\} \le \max\{nd+ e + f^*,nc+ f+e^*\}$$
for $n \ge \lin(I)+\lin(I)$. Therefore,
$$\max_{i \in [1,n-2]}\{\reg I^{n-i-1} + \reg J^i\} \le \max\{(n-1)d + e+ f^*,(n-1)c+ f+ e^*\}$$ 
for $n \ge \lin(I)+\lin(I)+1$. These bounds imply that
$$\reg R/(I+J)^{n-1} \le \max\{nd+ e+ f^* ,nc + f+ e^*\}-2.$$
\par
On the other hand, by Proposition \ref{r-quotient} we have
$$\reg (I+J)^{n-1}/(I+J)^n = \max\{(n+1)d + e+ f^*,(n+1)c+ f + e^*\}-2.$$
for $n \ge \lin(I)+\lin(I)$. Since  $c\ge d\ge 2$, it follows that 
$$\reg R/(I+J)^{n-1} +1 < \reg (I+J)^{n-1}/(I+J)^n$$
for $n \ge \lin(I)+\lin(I)+1$. Therefore, from the short exact sequence
$$0 \to (I+J)^{n-1}/(I+J)^n \to R/(I+J)^n \to R/(I+J)^{n-1}\to 0$$
we get $\reg R/(I+J)^n = \reg (I+J)^{n-1}/(I+J)^n$ for $n \ge \lin(I)+\lin(I)+1$.
\end{proof}

\begin{corollary}
If $c =d$ then $\lin(I+J) \le \lin(I)+\lin(J)+1$.
\end{corollary}



\begin{thebibliography}{9999}
\bibitem{ACH} A. Bagheri, M. Chardin and H.T. H\`a, The eventual shape of Betti tables of powers of ideals. Math. Res. Lett. {\bf 20} (2013), no. 6, 1033--1046.
\bibitem{BHH} S. Bandari, J. Herzog, and T. Hibi, Monomial ideals whose depth function has any given number of strict local maxima. Ark. Mat. {\bf 52} (2014), 11--19.
\bibitem{Be} D. Berlekamp, Regularity defect stabilization of powers of ideals. Math. Res. Lett. {\bf 19} (2012),  no. 1, 109--119.
\bibitem{Br} M. Brodmann, The asymptotic nature of the analytic spread. Math. Proc. Cambridge Philos. Soc. {\bf 86} (1979), 35--39.
\bibitem{BH} W. Bruns and J. Herzog, Cohen-Macaulay rings. Cambridge Studies in Advanced Mathematics 39, Cambridge University Press, Cambridge 1993.
\bibitem{Ch} M. Chardin, Powers of ideals and the cohomology of stalks and fibers of morphisms. Algebra Number Theory {\bf 7} (2013), no. 1, 1--18.
\bibitem{Co} A. Conca,  Regularity jumps for powers of ideals. Lect. Notes Pure Appl.
Math. {\bf 244} (2006), 21--32.
\bibitem{CHT} S.D. Cutkosky, J. Herzog and N.V. Trung, Asymptotic behaviour of the Castelnuovo-Mumford regularity. Compositio Math. {\bf 118} (1999), no. 3, 243--261.
\bibitem{E} D. Eisenbud, Commutative Algebra: with a View Toward Algebraic Geometry. Springer-Verlag, New York, 1995.
\bibitem{EG} D. Eisenbud and S. Goto, Linear free resolutions and minimal multiplicity, J. Algebra {\bf 88} (1984), 89--133.
\bibitem{EU} D. Eisenbud and B. Ulrich, Notes on regularity stabilization. Proc. Amer. Math. Soc. {\bf 140} (2012), no. 4, 1221--1232.
\bibitem{EH} D. Eisenbud and J. Harris, Powers of ideals and fibers of morphisms.
Math. Res. Lett. {\bf 17} (2010), no. 2, 267--273.
\bibitem{GW} S. Goto and K Watanabe, On graded rings I, Math. Soc. Japan 30 (1978), 179--212.
\bibitem{Ha} H.T. H\`a, Asymptotic linearity of regularity and a*-invariant of powers of ideals. Math. Res. Lett. {\bf 18} (2011), no. 1, 1--9.
\bibitem{HaS} H.T. H\`a and M. Sun, Squarefree monomial ideals that fails the persistence property and non-increasing depth, Acta Mathematica Vietnamica {\bf 40} (2015), no. 1, 125--138.
\bibitem{HH} J. Herzog and T. Hibi, The depth of powers of an ideal. J. Algebra {\bf 291} (2005), no. 2, 534--550.
\bibitem{HQ} J. Herzog and A.A. Qureshi, Persistence and stability properties of powers of ideals, 
J. Pure Appl. Algebra {\bf 229} (2015), 530--542.
\bibitem{HTT} J. Herzog, Y. Takayama, N. Terai, On the radical of a monomial ideal. Arch. Math. {\bf 85} (2005), 397--408.
\bibitem{HV} J. Herzog and M. Vladiou, Squarefree monomial ideals with constant depth function. J. Pure Appl. Algebra {\bf 217} (2013), no. 9, 1764--1772.
\bibitem{HT} L. T. Hoa and N. D. Tam, On some invariants of a mixed product of ideals. Arch. Math. {\bf 94} (2010), no. 4, 327--337.
\bibitem{Ko} V. Kodiyalam, Asymptotic behaviour of Castelnuovo-Mumford regularity.  Proceedings of Amer. Math. Soc. {\bf 128}, no. 2, (1999), 407--411.
\bibitem{NV} L.D. Nam and M. Varbaro, When does the depth stabilize soon?, J. Algebra {\bf 445} (2016), 181--192.
\bibitem{Sw} I.~Swanson, Powers of ideals.~Primary decompositions, Artin-Rees
lemma and regularity. Math. Ann. {\bf 307} (1997), 299--313.
\bibitem{TT1} N. Terai and N.V. Trung, Cohen-Macaulayness of large powers of Stanley-Reisner ideals. Advances in Mathematics {\bf 229} (2012), 711--730.
\bibitem{TT2} N. Terai and N.V. Trung, On the associated primes and the depth of the second power of squarefree monomial ideals. J. Pure Appl. Algebra {\bf 218} (2014), no. 6, 1117--1129.
\bibitem{TW} N.V. Trung and H-J. Wang, On the asymptotic linearity of Castelnuovo-Mumford regularity. J. Pure Appl. Algebra {\bf 201} (2005), 42--48.
\bibitem{Tr1} T.N. Trung, Stability of associated primes of integral closures of monomial ideals. J. Combin. Theory Ser. A {\bf 116} (2009), 44--54.
\bibitem{Tr} T.N. Trung, Stability of depth of power of edge ideals. Preprint 2013, to appear in J. Algebra.
\end{thebibliography}
\end{document}